\documentclass[reqno]{amsart}
\usepackage{amsmath, amssymb, amsthm, epsfig}
\usepackage{hyperref, latexsym}
\usepackage{url}
\usepackage{color}
\usepackage{fullpage}
\usepackage{setspace}
\usepackage{graphics}

\allowdisplaybreaks

\def\today{\ifcase\month\or
	January\or February\or March\or April\or May\or June\or
	July\or August\or September\or October\or November\or December\fi
	\space\number\day, \number\year}

\newtheorem{theorem}{Theorem}

\newtheorem{lemma}[theorem]{Lemma}

\newtheorem{corollary}[theorem]{Corollary}
\theoremstyle{definition}

\theoremstyle{remark}
\newtheorem{remark}[theorem]{Remark}

\newcommand{\de}[0]{\mathrel{\mathop:}=}
\newcommand{\ie}[0]{\mathrm{i}}
\newcommand{\dif}[1]{\mathrm{d}#1}

\newcommand{\R}{\mathbb{R}}
\newcommand{\N}{\mathbb{N}}

\newcommand{\p}{\varphi}

\newcommand{\hh}{\tfrac12}

\newcommand{\al}{\alpha}
\newcommand{\be}{\beta}

\newcommand{\si}{\sigma}

\begin{document}
	
	\title[Conditional estimates for the logarithmic derivative of Dirichlet $L$-functions]{Conditional estimates for the logarithmic derivative \\ of Dirichlet $L$-functions}
	\author[A.~Chirre, M.~V.~Hagen and A.~Simoni\v{c}]{Andr\'es Chirre, Markus Val{\aa}s Hagen and Aleksander Simoni\v{c}}
	\address{Departamento de Ciencias - Sección Matemáticas, Pontificia Universidad Católica del Perú, Lima, Perú}
	\email{cchirre@pucp.edu.pe}
	\address{Department of Mathematical Sciences, Norwegian University of Science and Technology, NO-7491 Trondheim, Norway}
	\email{markus.v.hagen@ntnu.no}
	\address{School of Science, The University of New South Wales (Canberra), ACT, Australia}
	\email{a.simonic@adfa.edu.au}
	\allowdisplaybreaks
	\numberwithin{equation}{section}

	\begin{abstract}
		Assuming the Generalized Riemann Hypothesis, we establish explicit bounds in the $q$-aspect for the logarithmic derivative $\left(L'/L\right)\left(\sigma,\chi\right)$ of Dirichlet $L$-functions, where $\chi$ is a primitive character modulo $q\geq 10^{30}$ and $1/2+1/\log{\log{q}}\leq\sigma\leq 1-1/\log\log q$. In addition, for $\sigma=1$ we improve upon the result by Ihara, Murty and Shimura (2009). Similar results for the logarithmic derivative of the Riemann zeta-function are given.
	\end{abstract}
	
	\maketitle
	\thispagestyle{empty}
	
	\section{Introduction}
		
	Let $\zeta(s)$ be the Riemann zeta-function and $s=\sigma+\ie t$, where $\sigma$ and $t$ are real numbers. It is well-known that for $\sigma>1$ the logarithmic derivative of the zeta-function admits an expansion into the Dirichlet series
	\begin{equation*}
	\frac{\zeta'}{\zeta}(s) = -\sum_{n=1}^{\infty}\frac{\Lambda(n)}{n^{s}},
	\end{equation*}
	where $\Lambda(n)$ is the von Mangoldt function. A classical result due to Littlewood ($1924$) asserts that the Riemann Hypothesis (RH) implies
	\begin{equation}
		\label{eq:Littlewood}
		\frac{\zeta'}{\zeta}(s) \ll \left((\log t)^{2-2\sigma}+1\right)\min\left\{\dfrac{1}{|\sigma-1|},\log\log t\right\}
	\end{equation}
	for $1/2+1/\log{\log{t}}\leq\sigma\leq3/2$ and $t$ large, see~\cite[Corollary 13.14]{MV}. In particular, $\left(\zeta'/\zeta\right)(s)\ll \left(\log{t}\right)^{2-2\sigma}$ for $1/2+\delta\leq\sigma\leq1-\delta$ and any fixed $\delta\in(0,1/4)$, and $\left(\zeta'/\zeta\right)(1+\ie t)\ll \log{\log{t}}$. The shape of these bounds has never been improved, and efforts have been placed into obtaining explicit constants for the main terms. Recently, an explicit bound for~\eqref{eq:Littlewood} has been given by Gon\c{c}alves and the first author in~\cite[Theorem 1]{chirreGoncalves}.
	
Similar results hold for $L$-functions. Let $\chi$ be a primitive character modulo $q$, and let $L(s,\chi)$ be the associated Dirichlet $L$-function. From now on we will assume that the Generalized Riemann Hypothesis (GRH) holds, i.e., all non-trivial zeros of $L(s,\chi)$ have the form $\rho_\chi = \frac{1}{2}+\ie\gamma_\chi$ for $\gamma_\chi\in\R$.
	
In $2009$, Ihara, Murty and Shimura~\cite[Corollary 3.3.2]{IharaMurtyShimura} proved under GRH that
	\begin{equation*}
		\left|\frac{L'}{L}\left(1,\chi\right)\right| \leq 2\log{\log{q}} + 2\left(1-\log{2}\right) + O\left(\frac{\log\log q}{\log{q}}\right).
	\end{equation*}
The implicit constant in the error term is stated explicitly. Note that $2\left(1-\log{2}\right)= 0.613\dots$. In the present paper we improve the latter result to the following.
	
	\begin{theorem}
		\label{thm:LogDerL}
		Assume the Generalized Riemann Hypothesis. Let $\chi$ be a primitive character modulo $q\geq 10^{30}$. Then
		\[
		\left|\dfrac{L'}{L}(1,\chi)\right|\leq 2\log \log q - 0.4989 + 5.91\frac{\left(\log{\log{q}}\right)^2}{\log{q}}.
		\]
		Also, for $q\geq 10^{153}$ we have $\left|\left(L'/L\right)(1,\chi)\right|\leq 2\log{\log{q}}$.
	\end{theorem}	
	
Numerical considerations concerning the extremal values of $M_q=\max_{\chi\neq\chi_{0}}\left\{\left|(L'/L)(1,\chi)\right|\right\}$ for odd prime numbers $q\leq 10^7$ can be found in~\cite{LanguascoRighi}, see also~\cite{LamzouriLanguasco} and~\cite{LanguascoEfficient}. Theorem~\ref{thm:LogDerL} may be used to obtain conditional and effective results in the distribution of prime numbers in arithmetic progressions. Here, one of the necessary ingredients is also an estimate on $b(\chi)$, i.e., the constant term in the Laurent expansion of $\left(L'/L\right)(s,\chi)$ at $s=0$. Bennett et al.~\cite[Proposition 1.12]{BennettAP} proved unconditionally that $\left|b(\chi)\right|\leq 0.2515q\log{q}$ for a Dirichlet character $\chi$ modulo $q\geq10^{5}$. From~\cite[Equation 10.35]{MV} we see that
$$
b(\chi)=-\dfrac{L'}{L}(1,\overline{\chi})-\log \dfrac{q}{2\pi}+\gamma,
$$
where $\overline{\chi}$ is the conjugate Dirichlet character and $\gamma$ is the Euler--Mascheroni constant.

\begin{corollary}
Assume the Generalized Riemann Hypothesis. Let $\chi$ be a primitive character modulo $q\geq 10^{30}$. Then $\left|b(\chi)\right|\leq \log{q}+2\log{\log{q}}-0.224$.
\end{corollary}

This improves the conditional estimate in \cite[Lemma 17]{EHP} for $q\geq 10^{30}$. Our next theorem provides effective and conditional estimates of the form~\eqref{eq:Littlewood} for Dirichlet $L$-functions of primitive characters modulo $q$ and $s=\sigma\in\R$ in the range\footnote{Following the same approach it is possible to get $\left(L'/L\right)(\sigma,\chi)\ll \log\log q$ in the range $1-1/\log\log q\leq \sigma\leq 1$.} $1/2+1/\log{\log{q}}\leq \sigma\leq 1-1/\log\log q$.
	
	\begin{theorem}
		\label{thm:LogDerL2}
		Assume the Generalized Riemann Hypothesis. Let $\chi$ be a primitive character modulo $q\geq 10^{30}$. For the range
			\begin{align}
				\label{20_31pm}
				\dfrac{1}{2}+\dfrac{1}{\log\log q}\leq \sigma\leq 1-\dfrac{1}{\log\log q},
			\end{align}
			we have
			\[
			\left|\dfrac{L'}{L}(\sigma,\chi)\right|\leq A_{\sigma}(\log q)^{2-2\sigma} - \dfrac{\sigma2^{1-\sigma}}{1-\sigma} + \frac{5.561\left(\log{q}\right)^{3-4\sigma}}{1-\sigma}
			+ \frac{0.306\left(\log{\log{q}}\right)^2}{2\sigma-1},
			\]
			where
			\begin{equation}
            \label{eq:Asigma}
			A_{\sigma} \de \dfrac{2(2\sigma-1)\left(1-\exp\left(-\frac{3(1-\sigma)}{2\left(2\sigma-1\right)}\right)\right)}{3(1-\sigma)^2} + 2.079.
			\end{equation}
	\end{theorem}	

Recently, explicit and conditional results on the logarithmic derivative of $L$-functions in the Selberg class of functions with a polynomial Euler product were published in~\cite{SimonicLfunctions}. However, they are worse than~\eqref{eq:Littlewood}. See also~\cite{Chandee,Sound,Tim,SimonicCS,SimonicSonRH} for other similar results, and the recent work of N. Paloj\"{a}rvi and the third author in \cite{SimonicPalo}.

	
	It was established in~\cite[Theorem 1]{Chi} by means of bandlimited functions that for
	\begin{align}
		\label{0_13am}
		\dfrac{1}{2}+\dfrac{1}{\log\log q} \leq  \sigma  \leq 1 - \frac{1}{\sqrt{\log\log{q}}},
	\end{align}
	and sufficiently large $q$ one has
	\begin{equation*}
		\left|\Re\left\{\dfrac{L'}{L}(\sigma,\chi)\right\}\right| \leq \left(\dfrac{-\si^2 + 3\si - 1}{\sigma(1-\sigma)}\right) (\log q)^{2-2\si}
+ O\left(\frac{\left(\log q\right)^{2-2\si}}{\left(\sigma-\hh\right)(1-\si)^2\log\log{q}}\right).
	\end{equation*}
	Following~\cite{chirreGoncalves}, we provide a similar estimate for the imaginary part.
	
	\begin{theorem}
		\label{0_30am}
		Assume the Generalized Riemann Hypothesis. Let $\chi$ be a primitive character modulo $q$. For sufficiently large $q$ in the range \eqref{0_13am} we have
		\begin{equation*}
			\left|\Im\left\{\dfrac{L'}{L}(\sigma,\chi)\right\}\right| \leq  \sqrt{\frac{2(-\si^2 + 3\si - 1)^2(-\si^2 + \si + 1)}{\si^3(1-\si)^2(2-\si)}}{(\log q)^{2-2\sigma}} + O\left(\frac{\left|\log\left(\sigma-\hh\right)\right|(\log q)^{2-2\si}}{\left(\sigma-\hh\right)(1-\si)^2\log\log{q}}\right).
		\end{equation*}
	\end{theorem}
	
In comparison to Theorem~\ref{thm:LogDerL}, we are able to provide a similar conditional estimate also for the Riemann zeta-function on the $1$-line.

\begin{theorem}
	\label{thm:LogDerZeta}
	Assume the Riemann Hypothesis. For $t\geq 10^{30}$, we have
	\[
	\left|\frac{\zeta'}{\zeta}\left(1+\ie t\right)\right| \leq 2\log \log t - 0.4989 + 5.35\frac{\left(\log{\log{t}}\right)^2}{\log{t}}.
	\]
	Also, for $t\geq 10^{137}$, we have $\left|\left(\zeta'/\zeta\right)(1+\ie t)\right|\leq 2\log{\log{t}}$.
\end{theorem}

In the proof of Theorems~\ref{thm:LogDerL} and~\ref{thm:LogDerL2} we are using a slightly modified version of Selberg's moment formula~\eqref{eq:selbergprin}, and the sum over the non-trivial zeros is estimated with a help of bandlimited majorants. By taking the same approach, we can recover under RH that
\begin{equation}
\label{eq:logzetageneral}
\left|\frac{\zeta'}{\zeta}(\sigma+it)\right| \leq A_{\sigma}\left(\log{t}\right)^{2-2\sigma} + O\left(\left(\log{t}\right)^{3-4\sigma}+\frac{1}{\left(2\sigma-1\right)^{3}}\right)
\end{equation}
in the range
\begin{align} \label{22_45pm}
	\dfrac{1}{2}+\frac{1}{\log\log t} \leq \sigma\leq 1-\frac{1}{\log \log t},
\end{align}
with $t$ sufficiently large and $A_{\sigma}$ defined by~\eqref{eq:Asigma}. Note that~\eqref{eq:logzetageneral} improves~\cite[Theorem 1]{chirreGoncalves} in the range~\eqref{22_45pm} and~\cite[Theorem 2]{chirreGoncalves} for $\sigma\geq0.51$. It would be interesting to prove~\eqref{eq:logzetageneral} for a larger family of $L$-functions (see \cite{SimonicPalo}).
	
The outline of this paper is as follows. In Section~\ref{sec:selberg} we revise Selberg's moment formula for Dirichlet $L$-functions and in Sections~\ref{sectionprimes} and~\ref{sectionzeros} we derive general estimates for the corresponding sums over prime numbers and non-trivial zeros, respectively. In Section~\ref{sec:threeproofs} we use these bounds to prove Theorems~\ref{thm:LogDerL} and~\ref{thm:LogDerL2}. The proof of Theorems~\ref{0_30am} and~\ref{thm:LogDerZeta} is provided in Sections~\ref{sec:ProofOfThmIm} and~\ref{secthm:LogDerZeta}, respectively.

	\section{The Selberg moment formula}
    \label{sec:selberg}
Selberg~\cite[Lemma 2]{SelbergOnTheNormal} discovered an interesting connection between the logarithmic derivative of the Riemann zeta-function and a special truncated Dirichlet series, which is also known as the Selberg moment formula. We apply this formula in the context of Dirichlet $L$-functions. Let $\chi$ be a primitive character modulo $q$ and let $L(s,\chi)$ be the associated Dirichlet $L$-function. We write $\mathfrak{a}=(1-\chi(-1))/2\in\{0,1\}$, depending on whether the character $\chi$ is even or odd. A variation of Selberg's formula~\cite[Equation 13.35]{MV} (see also \cite[Chapter 4, Theorem 1.7]{Joyner}) asserts for $x\geq2$ and $y\geq2$ that
	\begin{equation}		
		\label{eq:selbergprin}
		\frac{L'}{L}(s,\chi) = -\sum_{n\leq xy}\frac{\Lambda_{x,y}(n)\chi(n)}{n^s} + \frac{1}{\log{y}}\sum_{\rho_{\chi}}\frac{x^{\rho_{\chi}-s}-(xy)^{\rho_{\chi}-s}}{\left(\rho_{\chi}-s\right)^{2}} + \frac{1}{\log{y}}\sum_{n=0}^{\infty}\frac{x^{-2n-\mathfrak{a}-s}-(xy)^{-2n-\mathfrak{a}-s}}{\left(2n+\mathfrak{a}+s\right)^{2}}
	\end{equation}
	for $s\notin\left\{-2n-\mathfrak{a}\colon n\in\N_{0}\right\}$ and $s\neq\rho_{\chi}$, where
	\[
	\Lambda_{x,y}(n) \de \left\{\begin{array}{ll}
		\Lambda(n), & 1\leq n\leq x, \\
		\Lambda(n)\frac{\log{\frac{xy}{n}}}{\log{y}}, & x<n\leq xy,
	\end{array}
	\right.
	\]
	and the second sum runs over the non-trivial zeros $\rho_\chi$ of $L(s,\chi)$. Let $q\geq 10^{30}$ and consider the range
	\begin{align}  \label{1_01am}
		\dfrac{1}{2}+\dfrac{1}{\log\log q}\leq \sigma\leq 1.
	\end{align}
	In \eqref{eq:selbergprin} we take the parameters
	\begin{equation}	
		\label{eq:xy}
		y = \exp{\left(\frac{\lambda}{\sigma-\frac{1}{2}}\right)}, \quad x = y^{-1}\log^{2}{q},
	\end{equation}
	where $\lambda>0$ is chosen such that $x\geq 2$ and $y\geq 2$. Let us bound each term on the right-hand side of~\eqref{eq:selbergprin}. The first term is estimated easily by
	\begin{flalign}
		\label{21_04pm}
		\left|\sum_{n\leq xy}\frac{\Lambda_{x,y}(n)\chi(n)}{n^\sigma}\right| &\leq \sum_{n\leq x}\frac{\Lambda(n)}{n^\sigma} + \frac{1}{\log{y}}\sum_{x<n\leq xy}\frac{\Lambda(n)\log{\frac{xy}{n}}}{n^\sigma}  =: S_{x,y}(\sigma).
	\end{flalign}
	Since GRH holds, we estimate the second term in~\eqref{eq:selbergprin} as
	\begin{flalign}
		\label{0_29am}
		\frac{1}{\log{y}}\left|\sum_{\rho_{\chi}}\frac{x^{\rho_{\chi}-\sigma}-(xy)^{\rho_{\chi}-\sigma}}{\left(\rho_{\chi}-\sigma\right)^{2}}\right|
		&\leq \frac{\big(y^{\sigma-\frac{1}{2}}+1\big)(xy)^{\frac{1}{2}-\sigma}}{\log{y}}\sum_{\rho_{\chi}}\frac{1}{\left|\rho_{\chi}-\sigma\right|^{2}}\nonumber \\
		&\leq \frac{e^{\lambda}+1}{\lambda}\left(\log{q}\right)^{1-2\sigma}\sum_{\gamma_{\chi}}\frac{\sigma-\hh}{\left(\sigma-\hh\right)^{2}+\gamma_{\chi}^2},
	\end{flalign}
where the last sum runs over the ordinates of the non-trivial zeros of $L(s,\chi)$. Finally, using~\eqref{1_01am} we bound the last term as
	\begin{align}
		\label{15_28am} \frac{1}{\log{y}}\left|\sum_{n=0}^{\infty}\frac{x^{-2n-\mathfrak{a}-\sigma}-(xy)^{-2n-\mathfrak{a}-\sigma}}{\left(2n+\mathfrak{a}+\sigma\right)^{2}}\right|
		& \leq \frac{\left(y^{\sigma}+1\right)(xy)^{-\sigma}}{\log{y}}\sum_{n=0}^{\infty}\frac{1}{\left(2n+\frac{1}{2}\right)^{2}} \leq \dfrac{4.3\left(\sigma-\hh\right)(y^{\sigma}+1)}{\lambda(\log q)^{2\sigma}}.
	\end{align}
	The main difficulty remains to estimate two sums over primes in~\eqref{21_04pm}, and to estimate the sum over zeros~\eqref{0_29am}. We are going to do this in the following two sections.
	
	\section{The sum over prime numbers} \label{sectionprimes}

	In this section we bound the sum over the primes in~\eqref{21_04pm}. Firstly, we will provide an estimate when $\sigma=1$. We use the following lemma.
	
	\begin{lemma}
    \label{lem:primes}
		Let $\psi(x)=\sum_{n\leq x}{\Lambda(n)}$ and assume the Riemann Hypothesis. Then, for $x\geq 60$
\begin{align}  \label{16_142pm}
	\sum_{n\leq x}\frac{\Lambda(n)}{n} \leq \log{x} - \gamma  + \frac{\psi(x)-x}{x}+\dfrac{0.24}{\sqrt{x}}.
\end{align}
	In particular, for $x\geq 32$ we obtain
	\begin{equation}
	\label{eq:Lambda2}
	\sum_{n\leq x}\frac{\Lambda(n)}{n} \leq \log{x} - \gamma + 0.04\frac{\log^{2}{x}}{\sqrt{x}}.
\end{equation}
\end{lemma}
	
\begin{proof}
We follow partially the proof in~\cite[Lemma 2.2]{RamareExplicitLambda}. Using integration by parts and the fact\footnote{See~\cite[Proposition 3.4.4]{JamesonPNT}.} that $\int_{1}^\infty (\psi(u)-u)/u^{2}\dif{u}=-\gamma-1$, for $x\geq 2$ we have
		$$
\sum_{n\leq x}\frac{\Lambda(n)}{n} = \log{x} - \gamma + \frac{\psi(x)-x}{x} - \int_{x}^{\infty}\dfrac{\psi(u)-u}{u^2}\dif{u}.
		$$
Using Weil's explicit formula $\psi(u)=u-\sum_\rho\frac{u^\rho}{\rho}-\log 2\pi - \frac{1}{2}\log(1-u^{-2})$ when $u$ is not a prime power, we arrive at
			\begin{flalign}\label{21_30pm}
			\sum_{n\leq x}\frac{\Lambda(n)}{n} &= \log{x} - \gamma + \frac{\psi(x)-x}{x} - \sum_{\rho}\frac{x^{\rho-1}}{\rho(\rho-1)} +\int_{x}^{\infty}\dfrac{\log2\pi+\frac{1}{2}\log\left(1-u^{-2}\right)}{u^2}\dif{u}.
		\end{flalign}
Since RH holds, using\footnote{See~\cite[Equation 10.30]{MV}.} $\sum_{\rho}\frac{1}{|\rho(\rho-1)|}={2+\gamma-\log{4\pi}}$ and discarding the second part of the integral in the above expression, we arrive at
\begin{align}  \label{16_14pm}
	\sum_{n\leq x}\frac{\Lambda(n)}{n} \leq \log{x} - \gamma + \frac{\psi(x)-x}{x} +\dfrac{2+\gamma-\log{4\pi}}{\sqrt{x}} +\dfrac{\log2\pi}{x}.
\end{align}
When $x\geq 90$, using~\eqref{16_14pm} we arrive at~\eqref{16_142pm}. We check by computer that it also holds for $60\leq x\leq 90$. Now, using the explicit conditional bound\footnote{See~\cite[Equation 6.2]{SchoenfeldSharperRH}. Computer verification shows that it actually holds for $x\geq 59$. Moreover, we have that $\psi(x)-x\leq \frac{1}{8\pi}\sqrt{x}\log^{2}{x}$ for $x\geq 2$.}
$|\psi(x)-x |\leq \frac{1}{8\pi}\sqrt{x}\log^{2}{x}$, we conclude that~\eqref{eq:Lambda2} is true for $x\geq 4\cdot10^6$. We check by computer that it also holds for $32\leq x\leq 4\cdot10^6$.
\end{proof}

\begin{remark}
We remark that in~\cite[Lemma 2.1 and Lemma 2.2]{RamareExplicitLambda} there are some minor typos. For instance, the sign in front of $\log2\pi$. In~\cite[p. 81]{RamarePlatt} this typo is mentioned. We claim that Lemma 2.2  in \cite{RamareExplicitLambda} should be replaced by~\eqref{21_30pm} (after integration). This is valid unconditionally for all $x\geq 2$. Furthermore we claim that $\sum_{n\leq x}\Lambda(n)/n=\log{x}-\gamma+O^{\ast}(0.0067/\log{x})$ for $x\geq23$ in~\cite[Corollary on p.~114]{RamareExplicitLambda} is wrong, and has many counterexamples even for $x\geq10^4$.
\end{remark}

Now, assume that $x\geq 60$ and $\frac{1}{2}<\sigma\leq 1$. Using integration by parts one can see that
	\begin{align} \label{14_31pm}
	\frac{1}{\log{y}}&\sum_{x<n\leq xy}\frac{\Lambda(n)\log{\frac{xy}{n}}}{n^\sigma}  \nonumber  \\
	& \,\,\,\,\,\,= \dfrac{(xy)^{1-\sigma}}{\log y}\int_{1}^{y}\dfrac{\log u}{u^{2-\sigma}}\dif{u}-\dfrac{\psi(x)-x}{x^\sigma}-\dfrac{1}{\log y}\int_{x}^{xy}(\psi(u)-u)\left(\frac{1}{u^\sigma}\log\dfrac{xy}{u}\right)'\dif{u} \nonumber \\
	&\,\,\,\,\,\,\leq  \dfrac{(xy)^{1-\sigma}}{\log y}\int_{1}^{y}\dfrac{\log u}{u^{2-\sigma}}\dif{u}-\dfrac{\psi(x)-x}{x^\sigma}+\left(\sigma+\dfrac{1}{\log y}\right)\frac{\left(x^{\frac{1}{2}-\sigma}-(xy)^{\frac{1}{2}-\sigma}\right)\log^2(xy)}{8\pi\left(\sigma-\frac{1}{2}\right)}.
	\end{align}
Then, for $\sigma=1$, combining~\eqref{16_142pm},~\eqref{14_31pm} and recalling~\eqref{eq:xy} we have
\begin{align}
	\label{2_39am}
	S_{x,y}(1)& \leq \log x - \gamma  +\dfrac{0.24}{\sqrt{x}}+ \dfrac{\log y}{2}+\left(1+\dfrac{1}{\log y}\right)\frac{\left(x^{-\frac{1}{2}}-(xy)^{-\frac{1}{2}}\right)\log^2(xy)}{4\pi} \nonumber \\
	&= 		2\log\log q  - \gamma - \lambda + \left(\frac{(e^{\lambda}-1)(2\lambda+1)}{2\pi\lambda}\right)\frac{\left(\log{\log{q}}\right)^2}{\log{q}}+ \frac{0.24\,e^\lambda}{\log{q}}.
\end{align}
When $1/2<\sigma<1$, by integration by parts we see that
\begin{align} \label{14:32pm}
	\begin{split} 
\sum_{n\leq x}\frac{\Lambda(n)}{n^\sigma} & = \dfrac{\psi(x)-x}{x^{\sigma}} +\dfrac{x^{1-\sigma}-\sigma2^{1-\sigma}}{1-\sigma}+\sigma\int_{2}^{x}\dfrac{\psi(u)-u}{u^{\sigma+1}}\dif{u} \\
& \leq \dfrac{\psi(x)-x}{x^{\sigma}} +\dfrac{x^{1-\sigma}-\sigma2^{1-\sigma}}{1-\sigma}+\frac{\sigma}{8\pi}\int_{2}^{x}\frac{\log^{2}{u}}{u^{\sigma+\frac{1}{2}}}\dif{u} \\
& \leq \dfrac{\psi(x)-x}{x^{\sigma}} +\dfrac{x^{1-\sigma}-\sigma2^{1-\sigma}}{1-\sigma}+\frac{\sigma\log^2 x}{2^{\sigma+\frac{5}{2}}\pi\left(\sigma-\frac{1}{2}\right)}.
\end{split}
\end{align}
We directly combine~\eqref{14_31pm} and~\eqref{14:32pm}, with
\[
\frac{(xy)^{1-\sigma}}{\log{y}}\int_{1}^{y}\frac{\log{u}}{u^{2-\sigma}}\dif{u} = \frac{(xy)^{1-\sigma}-x^{1-\sigma}}{(1-\sigma)^2 \log{y}} -\frac{x^{1-\sigma}}{1-\sigma},
\]
to get
\begin{align}
		\label{3_13am}
		S_{x,y}(\sigma) &\leq \dfrac{(xy)^{1-\sigma}-x^{1-\sigma}}{(1-\sigma)^2\log y}-\dfrac{\sigma2^{1-\sigma}}{1-\sigma}+\frac{\sigma\log^2 x}{2^{\sigma+\frac{5}{2}}\pi\left(\sigma-\frac{1}{2}\right)}+\left(\sigma+\dfrac{1}{\log y}\right)\frac{\left(x^{\frac{1}{2}-\sigma}-(xy)^{\frac{1}{2}-\sigma}\right)\log^2(xy)}{8\pi\left(\sigma-\frac{1}{2}\right)} \nonumber \\
		& \leq B_{\sigma,\lambda}(\log q)^{2-2\sigma}-\dfrac{\sigma2^{1-\sigma}}{1-\sigma}+\frac{\sigma(\log\log q)^2}{2^{\sigma-\frac{1}{2}}\pi\left(2\sigma-1\right)}+\left(\dfrac{2\sigma}{2\sigma-1}+\dfrac{1}{\lambda}\right)\left(\dfrac{e^\lambda-1}{2\pi}\right)\dfrac{(\log\log q)^2}{(\log q)^{2\sigma-1}}.
	\end{align}
Here $B_{\sigma,\lambda}$ is given by
\begin{equation} \label{19_16pm}
B_{\sigma,\lambda}=\dfrac{(2\sigma-1)\left(1-\exp\left(-\frac{2\lambda(1-\sigma)}{2\sigma-1}\right)\right)}{2\lambda(1-\sigma)^2}.
\end{equation}

\section{The sum over the non-trivial zeros}
\label{sectionzeros}

In this section we obtain an explicit upper bound for the sum in~\eqref{0_29am} over the non-trivial zeros of $L(s,\chi)$. Firstly, we are going to derive an estimate when $\sigma=1$. To do that, we will use the known constant $B(\chi)$ since\footnote{See~\cite[Equation 10.38]{MV}.}
$\Re\left\{B(\chi)\right\} = -\sum_{\rho_\chi}\Re\left\{{1}/{\rho_\chi}\right\}$, and assuming GRH we have
	\[
\sum_{\gamma_{\chi}}\frac{	\frac{1}{2}}{\frac{1}{4}+\gamma_{\chi}^2} = \left|\Re\{B(\chi)\}\right|.
	\]
	Using~\cite[Lemmas 2.3 and 2.4]{Sound}, one can deduce that\footnote{The functions $E_\mathfrak{a}(z)$ defined in~\cite[Lemmas 2.3]{Sound} satisfy $E_\mathfrak{a}(z)\leq 0$ for $z\geq 4$.}
	\[
	\left|\Re\{B(\chi)\}\right| \leq \left(1-\frac{1}{\sqrt{z}}\right)^{-2}\left(\frac{1}{2}\left(1-\frac{1}{z}\right)\log{\frac{q}{\pi}}+\log{z}\right)
	\]
	for $z\geq 4$. Choosing $z=\frac{1}{4}\log^{2}{q}$  and recalling that $q\geq 10^{30}$ we obtain
	\begin{align}
		\label{3_45am}
		\begin{split}
	\sum_{\gamma_{\chi}}\frac{\frac{1}{2}}{\frac{1}{4}+\gamma_{\chi}^2} & \leq  \left(\dfrac{1}{2}+\dfrac{2}{\log q -2}\right)\log{\dfrac{q}{\pi}}+2\left(1-\dfrac{2}{\log q}\right)^{-2}\log\log q\leq \dfrac{1}{2}\log{\dfrac{q}{\pi}}+ 2.6\log\log q.
\end{split}	
\end{align}
For the range~\eqref{20_31pm} we proceed in a different way. Let $a=\sigma-\hh$ and let $f_a:\R\to\R$ be the function
\begin{align}  \label{22_55pm}
	f_{a}(x)=\dfrac{a}{a^2+x^2}.
	\end{align}
We want to estimate $\sum_{\gamma_\chi}f_a(\gamma_\chi)$ and the classical machinery to bound this sum is the Guinand--Weil explicit formula for Dirichlet $L$-functions~\cite[Lemma 4]{Chi} and for the Riemann zeta-function~\cite[Lemma 8]{CChiM}.
	
	\begin{lemma}
		\label{Guinand-weil}
		Let $q$ be a positive integer and $\chi$ be a primitive character modulo $q$. Let $h(s)$ be analytic in the strip $\left|\Im\{s\}\right|\leq \tfrac12+\varepsilon$ for some $\varepsilon>0$, and assume that $|h(s)|\ll(1+|s|)^{-(1+\delta)}$ as $\left|\Re\{s\}\right|\to\infty$, for some $\delta>0$. Assume GRH and that $h$ is a real-valued function. Then\footnote{Here $\widehat{h}$ denotes the Fourier transform of $h$, i.e., $\widehat{h}(\xi)=\int_{-\infty}^{\infty}h(u)e^{-2\pi\ie u\xi}\textup{d}u$.}
		\[
		\sum_{\gamma_\chi} h(\gamma_\chi)=\frac{1}{2\pi}\log\left(\dfrac{q}{\pi}\right)\widehat{h}(0)+\frac{1}{2\pi}\int_{-\infty}^\infty h(u)\Re\left\{\frac{\Gamma'}{\Gamma}\left(\frac{1}{4}+\frac{\mathfrak{a}}{2}+\frac{\ie u}{2}\right)\right\}\textup{d}u -\frac{1}{\pi}\sum_{n=2}^\infty\frac{\Lambda(n)}{\sqrt{n}}\Re\left\{\chi(n)\,\widehat{h}\left(\frac{\log n}{2\pi}\right)\right\}
		\]
	and
		\[
	\sum_{\gamma} h(\gamma)=2\,\Re\left\{\!h\left(\dfrac{\ie}{2}\right)\! \right\}-\frac{\log\pi}{2\pi}\widehat{h}(0)+\frac{1}{2\pi}\int_{-\infty}^\infty h(u)\Re\left\{\frac{\Gamma'}{\Gamma}\left(\frac{1}{4}+\frac{\ie u}{2}\right)\right\}\textup{d}u -\frac{1}{\pi}\sum_{n=2}^\infty\frac{\Lambda(n)}{\sqrt{n}}\Re\left\{\widehat{h}\left(\frac{\log n}{2\pi}\right)\right\},
	\]
where the sums on the left-hand sides run over the imaginary parts of the non-trivial zeros of $L(s,\chi)$ and $\zeta(s)$, respectively.
	\end{lemma}
	
	The function $f_a$ does not satisfy the conditions in Lemma~\ref{Guinand-weil}, and the key idea is to replace $f_a$ with certain explicit bandlimited majorants\footnote{To get the better bounds, we seek majorants that are extremal in the sense that they solve the Beurling--Selberg problem associated to $f_a$. This idea has been employed to estimates objects in the theory of the Riemann zeta-function and $L$-functions. See, for instance,~\cite{CCM2, CChiM, CS, Chi, GG}.} which are admissible for the classical Guinand--Weil explicit formula.
	
	In~\cite[Lemma 9]{CChiM}, it is proved that for any $\Delta>0$ the function
\begin{align}   \label{23_32pm}
	h(s)=h_{a,\Delta}(s)=\left(\dfrac{a}{a^2+s^2}\right)\left(\dfrac{e^{2\pi a\Delta}+e^{-2\pi a\Delta}-2\cos(2\pi\Delta s)}{\left(e^{\pi a\Delta}-e^{-\pi a\Delta}\right)^2}\right)
	\end{align}
	is a real entire function of exponential type $2\pi\Delta$ such that $f_a(u)\leq h(u)$ for all $u\in \R$, and its Fourier transform satisfies $\widehat{h}(\xi)\geq 0$ for all $|\xi|\leq \Delta$, $\widehat{h}(\xi)=0$ for all $|\xi|>\Delta$ and $\widehat{h}(0)=\pi\coth(\pi a\Delta)$. 
	Now, we follow the idea in \cite{chirreGoncalves}. Let $\Delta>0$ such that $\pi a\Delta\geq 1$. Because $h(u)\geq 0$ for all $u\in\R$, we have
\begin{align}
		\label{20_38pm}
		\sum_{\gamma_\chi} f_a(\gamma_\chi)\leq  \sum_{\gamma_\chi} h(\gamma_\chi) +\sum_{\gamma} h(\gamma),
		\end{align}
where the last sum runs over the ordinates $\gamma$ of the non-trivial zeros of $\zeta(s)$. From Lemma~\ref{Guinand-weil} we see that
	\begin{align} \label{18_18pm}
		\begin{split}
		\sum_{\gamma_\chi} h(\gamma_\chi) + \sum_{\gamma}h(\gamma) &= \frac{\log q}{2\pi}\,\widehat{h}(0)-\dfrac{\log \pi}{\pi}\widehat{h}(0) + 2\,\Re\left\{\!h\left(\dfrac{\ie}{2}\right)\! \right\} \\
		&+\frac{1}{2\pi}\int_{-\infty}^\infty  h(u)\Re\left\{\frac{\Gamma'}{\Gamma}\left(\frac{1}{4}+\frac{\mathfrak{a}}{2}+\frac{\ie u}{2}\right)\right\}  \dif{u} + \dfrac{1}{2\pi}\int_{-\infty}^{\infty}h(u)\Re\left\{\dfrac{\Gamma'}{\Gamma}\left(\dfrac{1}{4}+\dfrac{\ie u}{2}\right)\right\}\dif{u} \\ &-\frac{1}{\pi}\sum_{n=2}^{\infty}\frac{\Lambda(n)}{\sqrt{n}}\left(\Re\left\{\chi(n)\right\}+1\right)\widehat{h}\left(\frac{\log n}{2\pi}\right).
		\end{split}
	\end{align}
Since $\Re\left\{\chi(n)\right\}+1\geq 0$ and $\widehat{h}(\xi)\geq 0$, we discard the last sum in~\eqref{18_18pm}. Using the bound $\Re\left\{\left(\Gamma'/\Gamma\right)(s)\right\}\leq \log|s|$ for $\Re\{s\}\geq \frac{1}{4}$ (see~\cite[Lemma 2.3]{Chandee}) and the estimate
	\begin{align} \label{23_55pm}
	0\leq h(u)\leq \dfrac{a}{a^2+u^2}\left(\dfrac{e^{\pi a\Delta}+e^{-\pi a\Delta}}{e^{\pi a\Delta}-e^{-\pi a\Delta}}\right)^2 \leq \dfrac{1.725\,a}{a^2+u^2},
	\end{align}
one can bound the terms involving the gamma function as
	\begin{flalign*}
		\frac{1}{2\pi}\int_{-\infty}^\infty h(u)\Re\left\{\frac{\Gamma'}{\Gamma}\left(\frac{1}{4}+\frac{\mathfrak{a}}{2}+\frac{\ie u}{2}\right)\right\} \textup{d}u  &\leq \frac{1}{2\pi}\int_{-\infty}^\infty h(u)\log\left|\frac{3}{4}+\frac{\ie u}{2}\right|\textup{d}u \leq \frac{1}{2\pi}\int_{|u|\geq \frac{\sqrt{7}}{2}} h(u)\log\left|\frac{3}{4}+\frac{\ie u}{2}\right|\textup{d}u  \\
		&\leq \frac{1.725}{2\pi}\int_{|u|\geq \frac{\sqrt{7}}{2}} \dfrac{1}{1+u^2}\log\left|\frac{3}{4}+\frac{\ie u}{2}\right|\textup{d}u \leq 0.298.
	\end{flalign*}
Therefore, the contribution of these terms is at most $0.596$. Since $\widehat{h}(0)\geq\pi$, combining~\eqref{20_38pm},~\eqref{18_18pm} and using $\widehat{h}(0)=\pi\coth(\pi a\Delta)$ we get
\begin{equation*}
\sum_{\gamma_\chi} f_a(\gamma_\chi) \leq  \frac{\log q}{2\pi}\,\widehat{h}(0)+ 2\,\Re\left\{\!h\left(\dfrac{\ie}{2}\right)\! \right\} = \frac{\coth(\pi a\Delta)\log q}{2}+ \left(\dfrac{2a}{\frac{1}{4}-a^2}\right)\!\left(\dfrac{e^{\pi\Delta}+e^{-\pi\Delta}-e^{2\pi a\Delta}-e^{-2\pi a\Delta}}{\left(e^{\pi a\Delta}-e^{-\pi a\Delta}\right)^2}\right).
\end{equation*}
While discarding the negative term on the right-hand side and using $\pi a\Delta\geq 1$, it follows that
\begin{align*}
	\sum_{\gamma_\chi} f_a(\gamma_\chi) &\leq   \frac{\log q}{2}+\dfrac{e^{-2\pi a\Delta}\log q}{1-e^{-2}} + \left(\dfrac{2a}{\frac{1}{4}-a^2}\right)\left(\dfrac{e^{(1-2a)\pi\Delta}\left(1+e^{-2\pi\Delta}\right)}{\left(1-e^{-2}\right)^2}\right).
\end{align*}
We choose $\pi\Delta=\log\log q$ and recall that $a=\sigma-\hh$. Letting $\alpha=\left(1-e^{-2}\right)^{-1}$ and $\beta=1+(\log 10^{30})^{-2}$, we obtain
\begin{align*}
\sum_{\gamma_{\chi}}\frac{\sigma-\hh}{\left(\sigma-\hh\right)^{2}+\gamma_{\chi}^2} \leq\frac{\log q}{2}+\alpha\left(\dfrac{1+2\alpha\beta-\left(\sigma+\alpha\beta\sigma^{-1}\right)}{1-\sigma}\right)(\log q)^{2-2\sigma}.
\end{align*}
Clearly, in the range $\hh<\sigma<1$ we have that $\sigma+\alpha\beta\sigma^{-1}\geq 1+\alpha\beta$. Therefore,
\begin{align}  \label{16_12pm}
	\sum_{\gamma_{\chi}}\frac{\sigma-\hh}{\left(\sigma-\hh\right)^{2}+\gamma_{\chi}^2} \leq\frac{\log q}{2}+\dfrac{1.338}{1-\sigma}(\log q)^{2-2\sigma}.
\end{align}

\section{Proof of Theorems~\ref{thm:LogDerL} and~\ref{thm:LogDerL2}}
\label{sec:threeproofs}

    Having derived estimates for the terms in the Selberg moment formula~\eqref{eq:selbergprin} in the previous sections, we are now ready to prove Theorem ~\ref{thm:LogDerL} and Theorem~\ref{thm:LogDerL2}.
	
	\begin{proof}[Proof of Theorem~\ref{thm:LogDerL}]
		Letting $s=1$ in~\eqref{eq:selbergprin}, and combining~\eqref{0_29am},~\eqref{15_28am},~\eqref{2_39am} and ~\eqref{3_45am} it follows that
		\begin{align*}		
			\left|\frac{L'}{L}(1,\chi)\right|& \leq  2\log\log q  - \gamma - \lambda + \frac{e^{\lambda}+1}{2\lambda}
			+\left(\frac{(e^{\lambda}-1)(2\lambda+1)}{2\pi\lambda}\right)\frac{\left(\log{\log{q}}\right)^2}{\log{q}} \\
			&+\left(\frac{2.6\left(e^{\lambda}+1\right)}{\lambda}\right)\frac{\log{\log{q}}}{\log{q}}-\left(\frac{\left(e^{\lambda}+1\right)\log{\pi}}{2\lambda}-0.24\,e^\lambda\right)\frac{1}{\log{q}}+\left(\frac{2.15\left(e^{2\lambda}+1\right)}{\lambda} \right)\frac{1}{(\log{q})^2}.
		\end{align*}
We choose $\lambda=2.1862$ in order to minimize the constant term in the latter inequality. Note that this also implies $y=e^{2\lambda}\geq 2$, and that $q\geq 10^{30}$ implies $x=e^{-2\lambda}\log^2q\geq 60$. We arrive at
		\[
		\left|\dfrac{L'}{L}(1,\chi)\right|\leq 2\log \log q - 0.4989
		+ \left(3.091+\frac{11.776}{\log{\log{q}}}-\frac{0.455}{\left(\log{\log{q}}\right)^2}+\dfrac{78.906}{(\log\log q)^2\log q}\right)\frac{\left(\log{\log{q}}\right)^2}{\log{q}}.
		\]
This implies our desired results.
	\end{proof}
	
	\begin{proof}[Proof of Theorem~\ref{thm:LogDerL2}] We choose\footnote{The particular chooice $\lambda=\frac{3}{4}$ was chosen to beat Theorem~\ref{0_30am}.} $\lambda=\frac{3}{4}$. Recalling that $q\geq 10^{30}$ and using~\eqref{eq:xy}, we have $x\geq 60$ and $y\geq 2$ and the estimates obtained in the previous sections hold. Combining~\eqref{0_29am},~\eqref{15_28am},~\eqref{3_13am} and~\eqref{16_12pm} it follows that
\begin{align}\label{19_13pm}
		\left|\dfrac{L'}{L}(\sigma,\chi)\right|& \leq A_{\sigma}(\log q)^{2-2\sigma}-\dfrac{\sigma2^{1-\sigma}}{1-\sigma}+\frac{\sigma(\log\log q)^2}{2^{\sigma-\frac{1}{2}}\pi\left(2\sigma-1\right)}+\left(\dfrac{(14\sigma-4)\left(e^\frac{3}{4}-1\right)}{6\pi(2\sigma-1)}\right){(\log\log q)^2}{(\log q)^{1-2\sigma}} \nonumber \\
		& +\frac{4\left(e^{\frac{3}{4}}+1\right)1.338}{3(1-\sigma)}(\log q)^{3-4\sigma} +\dfrac{8.6\left(2\sigma-1\right)\left(e^{\frac{3\sigma}{2(2\sigma-1)}}+1\right)}{3(\log q)^{2\sigma}}
\end{align}		
with $A_\sigma=B_{\sigma,\frac{3}{4}}+\frac{2}{3}\left(e^{3/4}+1\right)$, where $B_{\sigma,\frac{3}{4}}$ is defined in~\eqref{19_16pm}. Finally, we bounded each term conveniently using~\eqref{20_31pm}. We remark that the factor $14\sigma-4$ is bounded by $10-{14}{(\log\log q)^{-1}}$ and this negative part cancels the last summand in the right-hand side of~\eqref{19_13pm}. Also, we use the fact that the function ${\sigma}{2^{-\sigma}}$ is increasing in $\sigma\in (\hh,1)$.
\end{proof}
	
\section{Proof of Theorem~\ref{0_30am}}
\label{sec:ProofOfThmIm}

Since the proof closely follows~\cite{chirreGoncalves} (see also Section~\ref{sectionzeros}), we will highlight only the main differences.
	
\subsection{Bounds for $\Re\left\{(L'/L)'(s,\chi)\right\}$}

In this section we are going to establish a lower and an upper bound for $\Re\left\{(L'/L)'(s,\chi)\right\}$, where $s=\sigma+\ie t$ with $\hh<\sigma<1$ and $|t|\leq \hh$. Taking real part of the derivative of the partial fraction decomposition of $L'/L$, see~\cite[Equation 10.37]{MV}, using a classical estimate for $\left(\Gamma'/\Gamma\right)'$ and using GRH, we arrive at
	\begin{align} \label{17_48pm}
	\Re\bigg\{\left(\frac{L'}{L}\right)'(s,\chi)\bigg\} = \sum_{\gamma_\chi} g_{a}(t-\gamma_\chi)+\sum_{\gamma } g_{a}(\gamma) + O(1),
\end{align}
	where $a=\sigma-\hh$, the function $g_a\colon\R\to\R$ is defined by $g_{a}(x)=\frac{x^2-a^2}{(x^2+a^2)^2}$, and the second sum runs over the imaginary part of the zeros of $\zeta(s)$. Note that we can add this sum here since $\sum_\gamma 1/\gamma^2<\infty$. 
	
	Now, we replace the function $g_a$ by the bandlimited majorants and minorants described in~\cite[Lemmas 6, 7, 8]{chirreGoncalves}. Let $\Delta\geq 1$ be a parameter such that $\pi a\Delta\geq 1$. The minorant function $m\de m_{a,\Delta}$ help us to derive the desired lower bound. We have that $m(u)\leq g_a(u)$ for all $u\in\R$, $m(u)=O\left((u^2+a^2)^{-1}\right)$ and $m\left(\frac{\ie}{2}\right)={2\pi\Delta a\,e^{(1-2a)\pi\Delta}}\left(a^2-\frac14\right)^{-1} + O\left({e^{(1-2a)\pi\Delta}\left(a-\frac{1}{2}\right)^{-2}}\right)$. Moreover $\widehat{m}(\xi)\leq 0$ for all $\xi\in\R$ and $\widehat{m}(0)=-4{\pi^2\Delta\,e^{-2a\pi \Delta}}+O\left({\Delta\,e^{-4a\pi\Delta}}\right)$. Then, we apply Lemma~\ref{Guinand-weil} as in~\eqref{18_18pm}. Since $|t|\leq \hh$, by Stirling's formula the terms with $\Gamma'/\Gamma$ are $O\left(1/a^2\right)$. Choosing $\pi\Delta=\log \log q$ we conclude that
\begin{align}  \label{1_07am}
	\Re\left\{\!\left(\frac{L'}{L}\right)'(s,\chi)\!\right\} \geq -\left(\frac{-2\sigma^2+6\sigma-2}{\sigma(1-\sigma)}\right)\log \log q \,(\log q)^{2-2\sigma}+ O\left(\dfrac{(\log q)^{2-2\sigma}}{\left(\sigma-\hh\right)(1-\sigma)^2}\right)
\end{align}
	for $\left(\sigma-\frac{1}{2}\right)\log \log q \geq 1$. Similarly, using the majorant one can get  in the range \eqref{0_13am} the upper bound
\begin{align}  \label{1_08am}
	\Re\left\{\!\left(\frac{L'}{L}\right)'(s,\chi)\!\right\} \leq \left(\frac{-2\sigma^2+2\sigma+2}{\sigma(1-\sigma)}\right)\log \log q \,(\log q)^{2-2\sigma}+ O\left(\dfrac{(\log q)^{2-2\sigma}}{\left(\sigma-\hh\right)(1-\sigma)^2}\right).
	\end{align}
	
	\begin{proof}[Proof of Theorem~\ref{0_30am}]
		Define the function $\p(t)=-\log{\left|L(s,\chi)\right|}$. Note that $
		\p'(t)=\Im\left\{\left(L'/{L}\right)(s,\chi)\right\}$, and $\p''(t)=\Re\left\{\left(L'/{L}\right)'(s,\chi)\right\}$. Let $|t|\leq 1/2$ and $q$ be sufficiently large. Denoting by $-\beta$ and $\alpha$ the right-hand sides in~\eqref{1_07am} and~\eqref{1_08am} respectively, we write $-\be\leq \p''(t)\leq\al$. Also, from~\cite[Theorem 1]{Chi} one can get the bounds $-\gamma \leq  \p(t)\leq \gamma$, where
		\[	
		\gamma = \left(\dfrac{-\si^2 +3\si - 1}{2\sigma(1-\sigma)}\right) \dfrac{(\log q)^{2-2\sigma}}{\log\log q} +\dfrac{c\left|\log\left(\sigma-\hh\right)\right|(\log q)^{2-2\sigma}}{\left(\sigma-\hh\right)(1-\sigma)^2(\log\log q)^2}.
		\]
		Let $|h|\leq 1/4$. An application of the mean value theorem gives that
		\[
		\p'(0)- \p'(-h) = \p''(h^*)\,h  \leq  \max\{h,0\}\al + \max\{-h,0\}\be.
		\]
		Here $h^* \in [-h,0]$ or $[0,-h]$. Averaging in $h$ in the interval $\left[-\nu (1-A), \nu A\right]$, we obtain
$\p'(0) \leq\frac{2\gamma}{\nu} + \frac{\nu}{2}\left(A^2 \al + (1-A)^2 \be\right)$,
		for $0<\nu<\frac{1}{4}$ and $0<A<1$. We minimize the right-hand side of the above expression by choosing $\nu=2\sqrt{{(\al^{-1}+\be^{-1})\gamma}}$ and $A={\be}{(\al+\be)^{-1}}$. Note that we have $\nu<\frac{1}{4}$ as $q\to\infty$. We conclude that
$\p'(0) \leq 2\sqrt{{\al\be (\al+\be)^{-1}\gamma }}.$
		The proof of the lower bound for $\p'(0)$ is similar. This implies the desired result.
	\end{proof}
	
	\begin{remark}
		By using $\left|\Im\left\{\left(L'/L\right)(\sigma,\chi)\right\}\right|\leq \left|\left(L'/L\right)(\sigma,\chi)\right|$, one can observe that the bound for the imaginary part in Theorem~\ref{0_30am} actually is better than Theorem~\ref{thm:LogDerL2} only when $\sigma$ is  very close to $0.5$, about $0.536$. We believe that to improve Theorem~\ref{0_30am} we  should estimate $\Im\left\{\left(L'/L\right)(\sigma,\chi)\right\}$ directly, without using the interpolation argument mentioned above. In fact, it is possible to obtain a representation for $\Im\left\{\left(L'/L\right)(\sigma,\chi)\right\}$ as in~\eqref{17_48pm}, and then one gets a function for which the Beurling--Selberg problem needs to be solved. However, for this specific function, the Beurling--Selberg problem is a hard problem and it is still open.
		\end{remark}
	
	\section{Proof of Theorem~\ref{thm:LogDerZeta}} 	
    \label{secthm:LogDerZeta}

	\begin{proof}[Proof of Theorem~\ref{thm:LogDerZeta}]
		We use Selberg's moment formula for the Riemann zeta-function~\cite[Equation (13.35)]{MV} with $s=1+\ie t$ and $y=e^{2\lambda}$ and $x=y^{-1}{\log^2t}$:
			\begin{equation}		
			\label{eq:selbergprinzeta}
			\frac{\zeta'}{\zeta}(s) = -\sum_{n\leq xy}\frac{\Lambda_{x,y}(n)}{n^s} + \frac{1}{\log{y}}\sum_{\rho}\frac{x^{\rho-s}-(xy)^{\rho-s}}{\left(\rho-s\right)^{2}} + \frac{1}{\log{y}}\sum_{n=1}^{\infty}\frac{x^{-2n-s}-(xy)^{-2n-s}}{\left(2n+s\right)^{2}}-\frac{x^{1-s}-(xy)^{1-s}}{\log y\left(1-s\right)^{2}}.
		\end{equation}
Note that the first sum is bounded exactly as in~\eqref{2_39am} replacing $q$ by $t$. Bounding as in~\eqref{0_29am}, the sum over the zeros in the right-hand side of~\eqref{eq:selbergprinzeta} is bounded by $\left(e^\lambda+1\right){\lambda^{-1}}(\log t)^{-1}\sum_{\gamma}f_{\frac{1}{2}}(t-\gamma)$, where the function $f_{\frac{1}{2}}$ is defined in~\eqref{22_55pm}. For $\Delta\geq 1$ and $a=\frac{1}{2}$ we apply Lemma~\ref{Guinand-weil} for the function $s\mapsto h(t-s)$ to get
\begin{align} \label{0_51am}
	\begin{split}
	\sum_{\gamma}f_{\frac{1}{2}}(t-\gamma)\leq\sum_{\gamma}h(t-\gamma)&\leq 2\left|h\left(t-\dfrac{\ie}{2}\right) \right|-\frac{\log\pi}{2\pi}\widehat{h}(0)+\frac{1}{2\pi}\int_{-\infty}^\infty h(u)\Re\left\{\frac{\Gamma'}{\Gamma}\left(\frac{1}{4}+\frac{\ie (t-u)}{2}\right)\right\}\textup{d}u \\
	& +\frac{1}{\pi}\sum_{n=2}^\infty\frac{\Lambda(n)}{\sqrt{n}}\left|\widehat{h}\left(\frac{\log n}{2\pi}\right)\right|,
	\end{split}
\end{align}
where $h(s)$ is the majorant function defined in~\eqref{23_32pm}. Since $t\geq 10^{30}$, for all $u\in\R$ we have
$$
\Re\left\{\frac{\Gamma'}{\Gamma}\left(\frac{1}{4}+\frac{\ie (t-u)}{2}\right)\right\}\leq \log\left|\frac{1}{4}+\frac{\ie (t-u)}{2}\right|=\log t + \dfrac{1}{2}\log\left(\dfrac{1}{16t^2}+\dfrac{1}{4}\left(1-\dfrac{u}{t}\right)^2\right)\leq \log t + \log(1+|u|).
$$
Thus, by~\eqref{23_55pm} with $a=\hh$, the third term in the right-hand side of~\eqref{0_51am} is bounded by $\frac{\widehat{h}(0)}{2\pi}\log t+0.541$. We bound the first term in the right-hand side of~\eqref{0_51am} using directly~\eqref{23_32pm}, and since $\widehat{h}(0)\geq\pi$ we arrive at
\begin{align*}
		\sum_{\gamma}f_{\frac{1}{2}}(t-\gamma)\leq \frac{\widehat{h}(0)}{2\pi}\log t +\frac{1}{\pi}\sum_{n=2}^\infty\frac{\Lambda(n)}{\sqrt{n}}\left|\widehat{h}\left(\frac{\log n}{2\pi}\right)\right|.
\end{align*}
To bound the above sum over primes, we remark from~\cite[Lemma 9]{CChiM} that $\widehat{h}(\xi)=0$ for all $|\xi|\geq \Delta$ and
\[
\widehat{h}(\xi)=\pi\left(\dfrac{e^{\pi (\Delta-|\xi|)}-e^{-\pi (\Delta-|\xi|)}}{e^{\pi\Delta}\left(1-e^{-{\pi\Delta}}\right)^2}\right) \quad \textrm{for all} \quad |\xi|\leq \Delta.
\]
Choosing $\pi\Delta=\log\log t$ and using~\eqref{16_142pm} it follows that
\begin{align*}
\frac{1}{\pi}\sum_{n=2}^\infty\frac{\Lambda(n)}{\sqrt{n}}\left|\widehat{h}\left(\frac{\log n}{2\pi}\right)\right| & = \dfrac{1}{\left(1-(\log t)^{-1}\right)^2}\sum_{n\leq (\log t)^{2}}\dfrac{\Lambda(n)}{n}-\dfrac{1}{(\log t)^{2}\left(1-(\log t)^{-1}\right)^2}\sum_{n\leq (\log t)^{2}}{\Lambda(n)} \nonumber \\
& \leq \dfrac{2\log\log t-\gamma-1+0.24(\log t)^{-1}}{\left(1-(\log t)^{-1}\right)^2}.
\end{align*}
Therefore, $\sum_{\gamma}f_{\frac{1}{2}}(t-\gamma)\leq 0.5\,{\log t}+ 2\log\log t$ (compare this estimate with~\eqref{3_45am}). We conclude that
\begin{align*}
	\left|\frac{1}{\log{y}}\sum_{\rho}\frac{x^{\rho-s}-(xy)^{\rho-s}}{\left(\rho-s\right)^{2}} \right|\leq \dfrac{e^\lambda+1}{2\lambda}+\dfrac{2\left(e^\lambda+1\right)\log\log t}{\lambda(\log t)}.
\end{align*}
Finally, we bound the last two terms in~\eqref{eq:selbergprinzeta} trivially. Therefore, taking $\lambda=2.1862$ and considering that $t\geq 10^{30}$ we obtain
\[
\left|\dfrac{\zeta'}{\zeta}(1+it)\right|\leq 2\log \log t - 0.4989
+ \left(3.091+\frac{9.06}{\log{\log{t}}}+\dfrac{2.137}{(\log\log t)^2}\right)\frac{\left(\log{\log{t}}\right)^2}{\log{t}}.
\]
Now the final result easily follows.
\end{proof}

	\section*{Acknowledgements}
	The authors would like to thank Olivier Ramar\'e for useful remarks concerning the proof of Lemma~\ref{lem:primes}, and Alessandro Languasco for pointing out a few references, as well as Tim Trudgian for taking time to read the manuscript. The project started when A.~C. was a postdoctoral fellow at NTNU. A.~C.~was supported by Grant 275113 of the Research Council of Norway, and M.~V.~H.~was supported by the Olav Thon Foundation through the StudForsk program. The authors also
	thankful to the anonymous referee for the valuable comments and suggestions.


\begin{thebibliography}{9999}
		
        \bibitem{BennettAP}
        M.~A. Bennett, G.~Martin, K.~O'Bryant, and A.~Rechnitzer,
        \newblock Explicit bounds for primes in arithmetic progressions,
        \newblock Illinois J. Math. 62 (2018), no.~1-4, 427--532.
		
		\bibitem{CCM2}
		E. Carneiro, V. Chandee and M. B. Milinovich,
		\newblock A note on the zeros of zeta and $L$-functions,
		\newblock Math. Z. 281 (2015), 315--332.
		
		\bibitem{CChiM}  E. Carneiro, A. Chirre and M. B. Milinovich,
		\newblock Bandlimited approximations and estimates for the Riemann zeta-function,
		\newblock Publ. Mat. 63 (2019), no. 2, 601--661.
		
	    \bibitem{Chandee}
		V. Chandee,
		\newblock Explicit upper bounds for $L$-functions on the critical line,
		\newblock Proc. Amer. Math. Soc. 137 (2009), no. 12, 4049–4063.
		
		\bibitem{CS}
		V. Chandee and K. Soundararajan,
		\newblock Bounding $|\zeta(\frac{1}{2}+it)|$ on the Riemann hypothesis,
		\newblock Bull. London Math. Soc. 43 (2011), no. 2, 243--250.
		
		\bibitem{Chi}  A. Chirre,
		\newblock A note on entire $L$-functions,
		\newblock Bull. Braz. Math. Soc. (N.S.) 50 (2019), no. 1, 67--93.
		
		\bibitem{chirreGoncalves}
		A. Chirre and F. Gon\c{c}alves,
		\newblock Bounding the log-derivative of the zeta-function,
		\newblock  Math. Z. 300 (2022), no. 1, 1041--1053.
		
		\bibitem{EHP}
		A-M. Ernvall-Hyt\"{o}nen and N. Paloj\"{a}rvi,
		\newblock Explicit bound for the number of primes in arithmetic
		progressions assuming the generalized {R}iemann hypothesis, 
		\newblock Math. Comp. 91 (2022), no. 335, 1317--1365.
		
		\bibitem{GG}
		D. A. Goldston and S. M. Gonek,
		\newblock A note on $S(t)$ and the zeros of the Riemann zeta-function,
		\newblock Bull. London Math. Soc. 39 (2007), 482--486.
		
		\bibitem{IharaMurtyShimura}
		Y.~Ihara, V.~K. Murty, and M.~Shimura,
		\newblock On the logarithmic derivatives of Dirichlet $L$-functions at $s=1$,
		\newblock Acta Arith. 137 (2009), no.~3, 253--276.
		
		\bibitem{JamesonPNT}
		G.~J.~O.~Jameson,
		\newblock {\it The prime number theorem},
		\newblock London Mathematical Society Student Texts 53, Cambridge University Press, Cambridge, 2003.
		
		\bibitem{Joyner}
		D.~Joyner,
		\newblock {\it Distribution theorems of {$L$}-functions}, 
		\newblock Pitman Research Notes in Mathematics Series 142, Longman Scientific \& Technical, Harlow; John Wiley \& Sons,
		Inc., New York, 1986.
		
        \bibitem{LamzouriLanguasco}
        Y.~Lamzouri and A.~Languasco,
        \newblock Small values of $|L'/L(1,\chi)|$,
        \newblock Exp. Math. 32 (2023), no.2, 362--377.

		\bibitem{Sound}
		Y. Lamzouri, X. Li, and K. Soundararajan,
		\newblock Conditional bounds for the least quadratic non-residue and related problems,
		\newblock Math. Comp. 84 (2015), no. 295, 2391--2412.

        \bibitem{LanguascoEfficient}
        A.~Languasco,
        \newblock Efficient computation of the Euler-Kronecker constants of prime cyclotomic fields,
        \newblock Res. Number Theory 7 (2021), no. 1, Paper No. 2, 22.

        \bibitem{LanguascoRighi}
        A.~Languasco and L.~Righi,
        \newblock A fast algorithm to compute the Ramanujan-Deninger gamma function and some number-theoretic applications,
        \newblock Math. Comp. 90 (2021), no. 332, 2899--2921.
		
		\bibitem{Tim}
		A. Languasco and T. S. Trudgian,
		\newblock Uniform effective estimates for $|L(1,\chi)|$,
		\newblock J. Number Theory 236 (2022), 245--260.
		
		\bibitem{MV}
		H. L. Montgomery and R. C. Vaughan,
		\newblock {\it Multiplicative Number Theory: I. Classical Theory},
		\newblock Cambridge Studies in Advanced Mathematics 97, Cambridge University Press, 2006.
		
		\bibitem{SimonicPalo}
		N. Palojärvi and A.~Simoni\v{c},
		\newblock Conditional estimates for $L$-functions in the Selberg class,
		\newblock preprint available at arXiv:2211.01121 (2022).
		
		\bibitem{RamareExplicitLambda}
		O.~Ramar\'{e},
		\newblock Explicit estimates for the summatory function of {$\Lambda(n)/n$} from the one of {$\Lambda(n)$},
		\newblock Acta Arith. 159 (2013), no.~2, 113--122.
		
		\bibitem{RamarePlatt}
		O.~Ramar\'{e} and D. J. Platt, \newblock Explicit estimates: from $\Lambda(n)$ in arithmetic progressions to $\Lambda(n)/n$,
		\newblock Exp. Math. 26 (2017), no. 1, 77–92.
		
		\bibitem{SchoenfeldSharperRH}
		L.~Schoenfeld,
		\newblock Sharper bounds for the Chebyshev functions $\theta(x)$ and $\psi (x)$. {II},
		\newblock Math. Comp. 30 (1976), no.~134, 337--360.
		
		\bibitem{SelbergOnTheNormal}
		A.~Selberg,
		\newblock On the normal density of primes in small intervals, and the difference between consecutive primes,
		\newblock Arch. Math. Naturvid. 47 (1943), no.~6, 87--105.
		
		\bibitem{SimonicLfunctions}
		A.~Simoni\v{c},
		\newblock Estimates for $L$-functions in the critical strip under GRH with effective applications,
		\newblock Mediterr. J. Math. 20 (2023), no. 2, paper no. 87, 24pp.
		
		\bibitem{SimonicCS}
		A.~Simoni\v{c},
		\newblock Explicit estimates for $\zeta(s)$ in the critical strip under the Riemann Hypothesis,
		\newblock Q. J. Math. 73 (2022), no. 3, 1055--1087.
		
		\bibitem{SimonicSonRH}
		A.~Simoni\v{c},
		\newblock On explicit estimates for $S(t)$, $S_1(t)$, and $\zeta(1/2+{\rm i}t)$ under the Riemann Hypothesis,
		\newblock J. Number Theory 231 (2022), 464--491.




        \end{thebibliography}

\end{document}